\documentclass{amsart}

\usepackage{amsmath,amssymb,amsthm,amsfonts,amscd,array,graphicx}
\usepackage{mathrsfs}

\usepackage[T2A]{fontenc}     
\usepackage[cp1251]{inputenc} 
\usepackage[english]{babel}    

\renewcommand{\C}{\mathbb{C}}
\newcommand{\K}{\mathbb{K}}
\newcommand{\Z}{\mathbb{Z}}
\newcommand{\Aut}{\mathrm{Aut}}
\newcommand{\X}{\mathrm{X}}

\newcommand{\A}{\mathcal{A}}
\newcommand{\T}{\mathbb{T}}

\newcommand\blfootnote[1]{%
	\begingroup
	\renewcommand\thefootnote{}\footnote{#1}%
	\addtocounter{footnote}{-1}%
	\endgroup
}

\newtheorem{theorem}{Theorem}[section]
\newtheorem{propos}[theorem]{Proposition}
\newtheorem{cor}[theorem]{Corollary}
\newtheorem{lem}[theorem]{Lemma}
\newtheorem{prop}[theorem]{Proposition}
\theoremstyle{definition}
\newtheorem{defin}[theorem]{Definition}
\newtheorem{remark}[theorem]{Remark}

\newtheorem{ex}[theorem]{Example}

\author{Anton Trushin}

\title{Automorphisms of rigid hypersurfaces with separable variables}

\date{11.07.2024}

\address{Lomonosov Moscow State University, Faculty of Mechanics and Mathematics, Moscow, Russia}
\address{Moscow Center for Fundamental and Applied Mathematics Moscow Russia}


\email{Trushin.ant.nic@yandex.ru}

\begin{document}
	
	\begin{abstract}
		Consider a polynomial $F$ such that each variable appears in exactly one monomial. The hypersurface defined by the polynomial $F$ is called a hypersurface with separable variables. A variety is called rigid if there are no nontrivial actions of the additive group of the ground field on it.	 If a variety is rigid, then it is known that in the automorphism group there exists a unique maximal torus. We describe the automorphism group of a rigid hypersurface with separable variables, in particular we show that it is a finite extension of the maximal torus $\T$.
	\end{abstract}
	
	\blfootnote{The work was supported by the Foundation for the Advancement of Theoretical Physics and Mathematics BASIS.}
	
	\maketitle
	
	\section{Introduction}
	Let $\K$ be an algebraically closed field of characteristic zero and let $\A = \K[x_1,\ldots,x_n]$. Let us consider an affine algebraic variety $X \subset \mathbb{A}^n$. The group of regular automorphisms $\Aut(X)$ in general is not an algebraic group.  Moreover, one can define the neutral component $\Aut(X)^0$, see \cite{Ra}, and it also can be not an algebraic group.
	In \cite{PZ} Perepechko and Zaidenberg proposed a conjecture that the neutral component of the automorphism group of a rigid variety is always a torus.	This conjecture is proved for toric varieties and varieties with the action of a torus of complexity 1, see \cite{BG}.
	It is known that if $\Aut(X)$ is algebraic and dimension of $X$ is greater than one, then $X$ is rigid, see for example \cite[Proposition 6.5]{Kr}. 
	Let $\mathbb{G}_a = (\K,+)$ be an additive group of the ground field. 
	Recall that a variety $X$ is called {\it rigid} if there are no nontrivial algebraic $\mathbb{G}_a$-actions on it.
	Rigid varieties were first studied by Miyanishi in \cite{Mi}.
	Then they were actively investigated in \cite{ML}. 
	There exists an ind-group structure on $\Aut(X)$, see \cite{Sh, FK}. 
	
	

	Let $X$ be an affine rigid variety. Then there is a subtorus $\T$ in $\Aut(X)$ containing all other subtori of $\Aut(X)$, see \cite[Theorem 1]{AG}.
	
	\begin{defin}
		Let us call a polynomial $F \in \A$ a {\it polynomial with separated variables} if each variable appears in exactly one monomial.
	\end{defin}
	
	Let us redefine the variables so that the polynomial $F$ takes the following form:
	\begin{equation}
		\begin{aligned}
			F = X_{11}^{l_{11}}\ldots X_{1n_1}^{l_{1n_1}} + X_{21}^{l_{21}}\ldots X_{2n_2}^{l_{2n_2}}+\ldots + X_{m1}^{l_{m1}}\ldots X_{mn_m}^{l_{mn_m}}+\\+Y_{11}^{q_1}+\ldots+Y_{1k_1}^{q_1}+\ldots+Y_{s1}^{q_s}+\ldots + Y_{sk_s}^{q_s},
		\end{aligned}
	\end{equation}\label{1}
	where $n_i > 1,\ q_1 > \ldots > q_s$ and put $L_i = l_{i1}+\ldots+l_{in_i}$.
	\begin{defin}
		Let us call an affine hypersurface $X\subset\mathbb{A}^n$ a hypersurface with {\it separable variables} if there are coordinates such that $X$ is the set of zeros of a polynomial with separated variables.
	\end{defin}
	
	\begin{lem}
		A rigid hypersurface with separated variables is irreducible.
	\end{lem}
	
	\begin{proof}
		Let the hypersurface include at least three monomials. Let us assume that it is reducible: $F = G\cdot H$. It is clear that $G$ and $H$ are not monomials. Moreover, there is a variable on which both $G$ and $H$ depend in a nontrivial way: let $G = Au^p + B, H = Cv^q + D$, where $u$ and $v$ are some variables, and $A, B, C, D$ are polynomials. Then $G\cdot H = ACu^pv^q + ADu^p + BCv^q + BD$. In this product, the monomials in $ACu^pv^q$ and $ADu^p$ include the variable $u$, which means it is not a hypersurface with separable variables.
		
		So, let's choose two variables $u$ and $v$ such that one of them is included in both $G$ and $H$. Then we fix the remaining variables so that $F$ takes the form $u^\alpha + v^\beta + 1 = Q$. If the polynomial $F$ is reducible, then~$Q$~is also reducible: $Q = P\cdot S$, a contradiction. 
	\end{proof}
	
	\section{Main result}
	
	Let us first consider linear automorphisms of the hypersurface defined by the polynomial $F = Y_1^{\alpha} + \ldots + Y_n^{\alpha}$.
	
	\begin{propos}\label{ds}
		If $\alpha \geqslant 2$ then the group of linear automorphisms of the hypersurface~$\mathbb{V}\left(Y_1^{\alpha} + \ldots + Y_n^{\alpha}\right)$ is isomorphic to $ S_n\rightthreetimes\left(\left(\Z/\alpha\Z\right)^{n-1}\times\mathrm{\mathbb{T}}\right)$.
	\end{propos}
	\begin{proof}
		Consider the Hessian Matrix of the function $F = Y_1^{\alpha} + \ldots + Y_n^{\alpha}$. It has the form: $$H(F) =\left(\frac{\partial^2 f(x)}{\partial y_i\partial y_j}\right) = \begin{pmatrix}
			(\alpha-1)\alpha Y_1^{\alpha-2}& 0 &\ldots & 0\\
			0& (\alpha-1)\alpha Y_2^{\alpha-2} &\ldots & 0\\
			\vdots& \vdots &\ddots & \vdots\\
			0& 0 &\ldots & (\alpha-1)\alpha Y_n^{\alpha-2}
		\end{pmatrix}$$ Thus, the Hessian of the function $f$ has the form $(\alpha-1)^n\alpha^ny_1^{\alpha-2}\ldots y_n^{\alpha{-2}}$.
		
		With a linear change of coordinates, the Hessian Matrix, as a matrix of quadratic form, will take the form $S^\top HS$, where $S$ is the transition matrix. The Hessian after the replacement has the form $\lvert S\rvert^2(\alpha-1)^n\alpha^ny_1^{\alpha-2}\ldots y_n^{\alpha{-2}}$, that is, its zeros are invariant. This means that for all $i$ the hyperplane $V_i = \{y_i = 0\}$ maps to the hyperplane $V_j = \{y_j = 0\}$ for some $j$. Accordingly, for all $k$ and $n$ the variety $V_k\cap V_n$ maps into $V_p\cap V_s$ for some $p$ and $s$. Similarly, the intersection of $n-1$ planes becomes the intersections of some $n-1$ planes. Thus $y_i \mapsto \lambda_{i}y_j$, and $\lambda_{p}^\alpha = \lambda_{q}^\alpha$ for all $p,q$.
	\end{proof}
	
	Further we will assume $X$ to be rigid and follow the five section from \cite{AG}.
	
	Let $X$ be a hypersurface with separable variables.\\
	Then one-dimensional tori act on $X$:
	\begin{enumerate}
		\item[1.] $\mathbb{T}_0:\\ t_0\cdot X_{ij} = t_0^{\frac{L_1\cdot\ldots\cdot L_m\cdot q_1\cdot\ldots\cdot q_s}{L_i}} X_{ij},\\ t_0\cdot Y_{ik} = t_0^{\frac{L_1\cdot\ldots\cdot L_m\cdot q_1\cdot\ldots\cdot q_s}{q_i}} Y_{ik}.$\\
		\item[2.] $\mathbb{T}_{ij}:\\ t_{ij}\cdot X_{i1} = t_{ij}^{l_{ij}} X_{i1};\\
		t_{ij}\cdot X_{ij} = t_{ij}^{-l_{i1}} X_{ij},\ j\neq 1,$\\
		the rest of the generators are stable.
	\end{enumerate}
	
	Let $\sigma$ be a cone spanned by the vectors of weights of the basis functions with respect to the torus $\T$, where $\T$ is generated by the one-dimensional tori described above.
	\begin{remark}\label{po}
		The one-parameter subgroup corresponding to $\mathbb{T}_0$ acts on all generators with a positive weight. Thus, the cone $\sigma$ is pointed, that is, it does not contain lines.
	\end{remark}
	\begin{remark}
		It is well known that the only rigid affine toric variety is the torus itself, see e.g. \cite{AKZ}.
		Since origin is a singular point, the variety $X$ is not toric.
		Thus the torus $\T$ is the maximal torus in the group $\Aut(X)$.
	\end{remark}

	\begin{lem}
		The function $X_{ij}$ maps to $\lambda X_{pq}$.
	\end{lem}
	\begin{proof}
		By \cite[Theorem 1]{AG}, any automorphism normalizes the torus $\T$. The weights $X_{ij}$ are the minimal system of cone-generating weights of the torus $\T$, consisting of primitive elements of the weight semigroup. In this case, according to remark \ref{po}, the cone $\sigma$ is pointed, which means that the system of weights $X_{ij}$ is preserved.
	\end{proof}
	\begin{cor}\label{x}
		If $X_{ij}$ maps to $\lambda X_{pq}$, then~$l_{ij} = l_{pq}$ and $n_i = n_p$, $(l_{i1},\ldots,l_{in_i})=(l_{p1} ,\ldots,l_{pn_p})$.
	\end{cor}
	
	\begin{proof}
		Let us consider a subvariety $A$ of $X$ given by the set of equations $X_{rt} = 1, Y_k = 1$, where $(r,t) \neq (p,q)$. Such a variety consists of $l_{pq}$ points, the coordinate $X_{pq}$ of which is found from the equation $X_{pq}^{l_{pq}} = const \neq 0$. If $X_{ij} \mapsto \lambda X_{pq}$ then the image of the subvariety~$A$ is a subvariety of $B$ which is given by set of equations $\lambda_{rt}X_{rt} = 1, \lambda_kY_k = 1$ and equation $\lambda X_{ij}^{l_{ij}} = const$. But there are exactly $l_{ij}$ points in the subvariety $B$, which means $l_{ij} = l_{pq}$.
		
		Now consider all variables except the variables included in the $i$-th monomial. Let us set them equal to zero. If under the automorphism these variables are included in all monomials, then the image of the subvariety $\mathbb{V}\left(X_{i1}^{l_{i1}}\ldots\X_{in_i}^{l_{in_i}}\right) $ is $\mathbb{V}\left(0\right)$, a contradiction. Now let several monomials consist entirely of the product of images of the variables included in the $i$-th monomial. Let $p$-th be one of them. Let us set all variables except the variables of this monomial equal to zero. Under an inverse automorphism, the images of all variables of the $p$-th are included in the $i$-th monomial. If at the same time the $i$-th monomial contains some other factors, then, similarly to the previous reasoning, the image of the subvariety~$\mathbb{V}\left(X_{p1}^{l_{p1}}\ldots\X_{ pn_p}^{l_{pn_p}}\right)$ is $\mathbb{V}\left(0\right)$. This means that the variables in the $p$-th monomial are images of the variables of the $i$-th monomial, hence $n_i = n_p$, and, by what was proved above, $(l_{i1},\ldots,l_{in_i})=(l_ {p1},\ldots,l_{pn_p})$.
		
		
	\end{proof}
	\begin{lem}\label{y}
		The function $Y_{ij}$ maps to $\sum\lambda_{iq} Y_{iq}$.
	\end{lem}
	
	\begin{proof}
		The functions $Y_{ij}$ under the action of the torus $\T$ have non-zero weight only with respect to the one-parameter subgroup corresponding to the torus $\mathbb{T}_0$. Thus, their vectors of weights in the cone lie on the line $\langle\sum\limits_{ij}l_{ij}v_{ij}\rangle$, where $v_{ij}$ is the vector of weights of the function $X_{ij} $ relative to the torus $\T$. By Corollary \ref{x}, this line is invariant under automorphisms.\\
		Let $w_{ij}$ be the vector of weights of the function $Y_{ij}$. Vector $w_{ij}$ is shorter then $w_{i+1j}$ since $q_i > q_{i+1}$. This means $Y_{ij}\mapsto\sum\limits_{q}\lambda_p Y_{iq} + F_j$, where $F_j \in \K[X_{11},\ldots,X_{mn_m},Y_{11},\ldots,Y_{i-1j}]$ is homogeneous with respect to the $\T_0$-action.
		
		Thus, setting the variables $X_{11},\ldots,X_{mn_m},Y_{11},\ldots,Y_{i-1j}$ equal to zero we obtain a linear automorphism of the variety $\mathbb{V}\left(Y_{i1}^{q_i} + \ldots + Y_{ik_i}^{q_i}\right)$. Now from Proposition~\ref{ds} we can assume that $Y_{ij}\mapsto Y_{ij} + F_j$.
		So, an automorphism of the variety $X$ induces a triangular automorphism $\varphi$ of the space. A triangular map is embedded in a unipotent group and thus induces the action of an additive group.
		Note that the lines that are the orbits of this action intersect with the variety $X$ at a countable number of points: let $x \in X$ then $\varphi^n(x) \in X$. Which means that such lines lie entirely in $X$ that is, there is an action of an additive group on $X$ if $F_j \neq 0$.
	\end{proof}
	
	Let $\tau$ be a permutation of the variables $X_{ij}$ and $Y_k$. Then $\tau$ acts naturally on the polynomial
	algebra $\K[X_{ij},Y_k]$. We say that $\tau$ is a permutation of a polynomial	$F$. All such permutations form a subgroup $P(F)$ in the symmetric group $S_n$. We call this subgroup the permutation group of the polynomial $F$.
	Let $\mathbb{H}$ is a subquasitorus of linear torus on $\mathbb{A}^n$ which is preserves zeroes of $F$.
	From Corollary \ref{x} and Lemma \ref{y} we obtain following theorem:
	
	\begin{theorem}
		Let $X$ be a rigid variety with separable variables, which is given by the equation $F = 0$, where $F$ is a polynomial with separated variables.
		Then the automorphism group of $X$ is isomorphic to $P(F)\rightthreetimes \mathbb{H}$.
	\end{theorem}
	
	Sufficient conditions for such varieties to be rigid are given in \cite[Example 4.6]{Kik}:
	\begin{prop}
		Assume $\sum l^{-1}_{ij} + \sum q^{-1}_i \leqslant \dfrac{1}{m + s - 2}$. Then $\A/(F)$ is rigid.
	\end{prop}
	\begin{ex}
		Let $F = X_1^{10}X_2^{11} + Y_1^{10} + Y_2^{10} + Y_3^{10}$ and $X = \mathbb{V}(F)$.\\
		Then $X$ is rigid since $\dfrac{1}{10} + \dfrac{1}{10} + \dfrac{1}{11} + \dfrac{1}{10} + \dfrac{1}{10} < \dfrac{1}{2}$.\\
		So $\mathrm{Aut}(X) = S_3\rightthreetimes\left(\left(\Z/10\Z\right)^{2}\times\mathrm{\mathbb{T}}\right)$.
	\end{ex}

\end{document}